\documentclass[10pt]{amsart}
\usepackage[linktocpage=true,colorlinks=true,citecolor=blue]{hyperref}
\usepackage{amsrefs, amssymb,amsfonts,amsthm}
\usepackage{enumerate}
\usepackage[applemac]{inputenc}
\usepackage{wrapfig}

\usepackage{url}

\DefineSimpleKey{bib}{myurl}

\newcommand\myurl[1]{\url{#1}}

\BibSpec{webpage}{%
  +{}{\PrintAuthors} {author}
  +{,}{ \textit} {title}
  +{}{ \parenthesize} {date}
  +{,}{ \myurl} {myurl}
  +{,}{ } {note}
  +{.}{ } {transition}
}

\usepackage{graphicx}
\usepackage{float}

\newtheorem{thm}{Theorem}
\newtheorem*{cor*}{Corollary}

\newtheorem{lem}{Lemma}

\theoremstyle{definition}

\theoremstyle{remark}


\DeclareMathOperator*{\Conv}{\raisebox{-0.6ex}{\scalebox{2.5}{$\ast$}}}

\title[Infinitely differentiable function]{An infinitely differentiable function with
compact support: Definition and Properties}
\author{J. Arias de Reyna}
\address{Univ.~de Sevilla \\
Facultad de Matem\'aticas \\
c/ Tarfia, sn
 \\
41012-Sevilla \\
Spain} 
\email{arias@us.es}

\thanks{ Received: November 5, 1980}

\begin{document}

\newcommand{\N}{{\mathbb N}}
\newcommand{\R}{{\mathbb R}}
\newcommand{\C}{{\mathbb C}}
\newcommand{\Z}{{\mathbb Z}}
\newcommand{\Q}{{\mathbb Q}}
\newcommand{\arctanh}{\mathop{\rm arctanh }}
\def\Re{\operatorname{Re}}
\def\Im{\operatorname{Im}}
\def\Hermite{\operatorname{\mathcal H}}
\def\supp{\operatorname{supp}}

\newfont{\cmbsy}{cmbsy10}
\newcommand{\Orden}{\mathop{\hbox{\cmbsy O}}\nolimits}


\maketitle

\section{Introduction.}

Infinitely differentiable functions of compact support defined on $\R$ play 
an important role in Analysis. Usually, one constructs   examples using an idea 
of Cauchy. For this example the derivatives are cumbersome. This problem makes
me search for a better example.

Looking at a rough plot of such a function and its derivative 
(see figure 1) I asked 
if it was possible that the derivative could be formed with two homothetic 
copies of the  same function translated conveniently.
So I posed  the following question:

\begin{wrapfigure}{r}{4.9cm}
\centering
\includegraphics[width=\hsize]{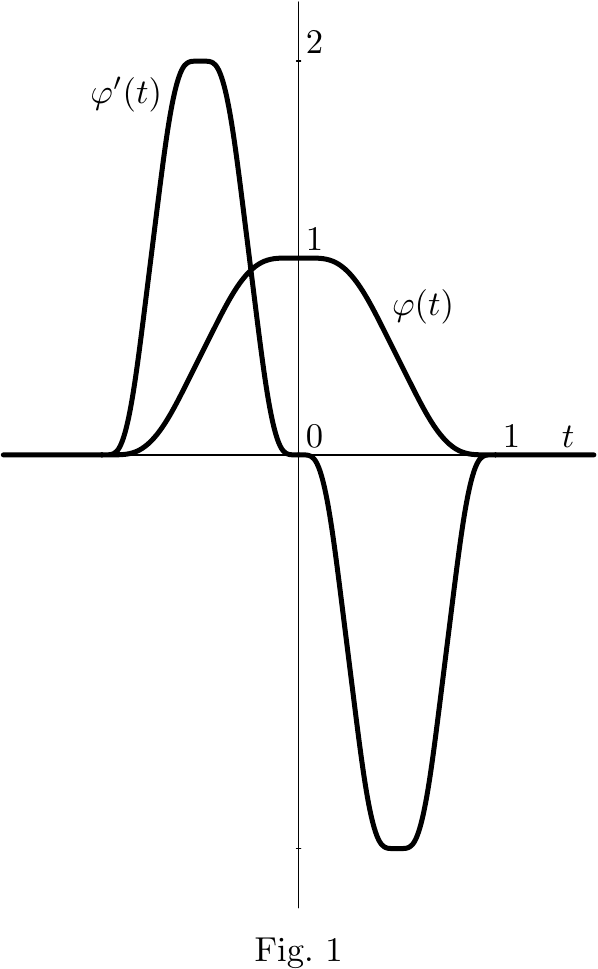}
\end{wrapfigure}
Does there exist a function $\varphi\in\mathcal D(\R)$ such that:
\begin{itemize}
\item[(a)] $\supp(\varphi)=[-1,1]$,
\item[(b)] $\varphi(t)>0$ for any $t\in(-1,1)$,
\item[(c)] $\varphi(0)=1$,
\item[(d)] and there is a constant $k>0$ such that for any $t\in\R$

$\displaystyle{\varphi'(t)=k\bigl(\varphi(2t+1)-\varphi(2t-1)\bigr)}$?

\end{itemize}
We will prove that there is a unique solution $\varphi$ satisfying the above 
conditions. For this unique solution the value of the constant  $k$ is $2$. No other value of 
$k$ gives a solution.

The function $\varphi$ has many other properties. It can be interpreted as 
a probability (theorem \ref{T:3}), $\varphi$ and some of its translates form a 
partition  of unity (theorem \ref{T:5}), its derivatives can be computed easily
(theorem \ref{T:4}), and the most notable, it is not a rational function but 
its values at dyadic points are rational numbers that are effectively computable.
Since its derivatives are related to the same function, not only the values 
of $\varphi$ but also those of its derivatives $\varphi^{(k)}(t)$ are rational number 
at dyadic points.

The only reference that we know about 
this function is a paper \cite{MR1501802} by Jessen and Wintner (1935) where
the function $\varphi$ is defined by means of its Fourier transform, as an 
example of an infinitely differentiable function, but Jessen and Wintner do 
not give any other property of this function.

\section{Existence and Unicity.}

\begin{thm}\label{T:1}
There is a unique  infinitely differentiable function with compact support $\varphi\colon\R\to\R$ and such that:
\begin{itemize}
\item[(a)] $\supp(\varphi)=[-1,1]$.
\item[(b)] $\varphi(t)>0$ for any $t$ in the open set $(-1,1)$.
\item[(c)] $\varphi(0)=1$.
\item[(d)] There is a constant $k>0$ such that for any $t\in\R$
\[
\varphi'(t)=k\bigl(\varphi(2t+1)-\varphi(2t-1)\bigr)
\]
\end{itemize}
and the constant $k$ appearing in (d) is necessarily equal to $2$.
\end{thm}
\begin{proof}
First, assuming that $\varphi$ exists, we will prove the unicity of $\varphi$ and that $k=2$.

Since $\varphi\in\mathcal D(\R)$ its Fourier transform is an entire function
\begin{equation}\label{E:1}
\widehat\varphi(z)=\int_\R \varphi(t)e^{-2\pi i t z}\,dt
\end{equation}
The Fourier transform of  $\varphi'(t)$, $\varphi(2t+1)$ and $\varphi(2t-1)$ are
\[2\pi i z \widehat\varphi(z),\quad e^{\pi i z}\widehat\varphi(\tfrac{z}{2}), 
\quad e^{-\pi i z}\widehat\varphi(\tfrac{z}{2})\]
respectively. Condition (d) yields
\begin{equation}\label{E:2}
\widehat\varphi(z)=\frac{k}{2}\frac{\sin\pi z}{\pi z}\widehat\varphi(\tfrac{z}{2}).
\end{equation}
By induction, we obtain from \eqref{E:2} that
\begin{equation}\label{E:3}
\widehat\varphi(z)=\Bigl(\frac{k}{2}\Bigr)^n\Bigl[\prod_{h=0}^n \frac{\sin\frac{\pi z}{2^h}}
{\frac{\pi z}{2^h}}\Bigr]\widehat\varphi\Bigl(\frac{z}{2^{n+1}}\Bigr).
\end{equation}
Conditions (a) and (b) imply that $\widehat\varphi(0)=\int\varphi(t)\,dt>0$, so that 
taking limits for $n\to\infty$ we obtain $k=2$ and 
\begin{equation}\label{E:4}
\widehat\varphi(z)=\widehat\varphi(0)\prod_{h=0}^\infty \frac{\sin\frac{\pi z}{2^h}}
{\frac{\pi z}{2^h}}.
\end{equation}
If there is a solution to our problem it is unique, because by the inversion formula 
\begin{equation}\label{E:5}
\varphi(t)=\int_\R\widehat\varphi(x)e^{2\pi i t x}\,dx
\end{equation}
and condition (c) will fix the value of the constant $\widehat\varphi(0)$. 

We will see later that (c) implies $\widehat\varphi(0)=1$, so that in what follows 
we will use $\widehat\varphi(z)$ to denote the function defined in \eqref{E:4}  assuming
$\widehat\varphi(0)=1$.

Now we will show that the solution $\varphi$ exists. We start from the function $\widehat\varphi(z)$
defined in \eqref{E:4}. Since the infinite product converges uniformly in compact sets, the function
$\widehat\varphi(z)$ is entire. Equation \eqref{E:2} may be used to expand it in power series
\begin{equation}\label{E:6}
\widehat\varphi(z)=\sum_{k=0}^\infty (-1)^k \frac{c_k}{(2k)!}(2\pi z)^{2k},
\end{equation}
where the $c_k$ are rational numbers defined by the recurrence
\begin{equation}\label{E:7}
(2k+1)2^{2k}c_k=\sum_{h=0}^k \binom{2k+1}{2h}c_h.
\end{equation}
From equation \eqref{E:7} we obtain that the numbers $c_k$ are positive. Also we have
\begin{equation}\label{E:8}
c_k=\frac{F_k}{(2k+1)(2k-1)\cdots 1}\prod_{n=1}^k (2^{2n}-1)^{-1},
\end{equation}
where $F_k$ are natural numbers, $F_0=1$, $F_1=1$, $F_2=19$, $F_3=2915$, $F_4=2\,788\,989$.

Using the known formulas
\[\frac{\sin z}{z}=\prod_{n=1}^\infty\cos\frac{z}{2^n},\quad \text{and}\quad 
\frac{\sin\pi z}{\pi z}=\prod_{n=1}^\infty \Bigl(1-\frac{z^2}{n^2}\Bigr),\]
we obtain 
\begin{equation}\label{E:9}
\widehat\varphi(z)=\prod_{m=1}^\infty \Bigl(\cos\frac{\pi z}{2^m}\Bigr)^m=\prod_{m=1}^\infty
\Bigl(1-\frac{z^2}{m^2}\Bigr)^{1+v_2(m)}, 
\end{equation}
where $v_2(m)$ is the greatest exponent such that $2^{v_2(m)}$ divides $m$. 

It is clear that $\widehat\varphi$ restricted to $\R$ is infinitely differentiable. We will 
show also that it is a rapidly decreasing function. 

Let $f(x)=(\sin x)/x$. For $x\in\R^*$, we have $|f(x)|\le 1$ and $|\sin x|\le 1$. For 
all $n$
\[|x^n\widehat\varphi(x)|=\Bigl|x^n \prod_{h=0}^\infty f(\pi x/2^h)\Bigr|\le \Bigl|x^n\prod_{h=0}^{n-1}
f(\pi x/2^h)\Bigr|\le 2^{\binom{n}{2}}\pi^{-n}.\]
It is easy to see that there is a constant $M_r\ge0$ for each $r\in\N$ such that 
\[|\partial^r f(\pi x/2^h)|\le \pi^r2^{-hr}M_r.\]
Applying the rule to differentiate an infinite product and the same idea 
used above to bound $|x^n\widehat\varphi(x)|$ we obtain
\begin{multline*}
|x^n \partial^r \widehat\varphi(x)|\le\\
\le \sum_S\frac{r!}{s_1!\cdots s_t!}\sum_H\Bigl|\prod_{i=1}^t \partial ^{s_i}f(\pi x/2^{h_i})\Bigr|\;
\Bigr|x^n\prod_{h\ne h_i}f(\pi x/2^h)\Bigr|\\
\le \sum_S\frac{r!}{s_1!\cdots s_t!}M_{s_1}\cdots M_{s_t}\Bigl(\sum_H\pi^r2^{-s_1h_1-\cdots-s_t h_t}\Bigr)2^{\binom{n+t}{2}}\pi^{-n}<\infty
\end{multline*}
where the sum extended to $S$ refers to all sets $\{s_1,\dots, s_t\}$ of natural numbers such 
that $s_1+\cdots+s_t=r$ and $s_i\ge1$ and the sum in $H$ to all sets $\{h_1,\dots, h_t\}$
of $t$ distinct  natural numbers. 

Once we have proved that $\widehat\varphi$ is a test function in Schwartz space we define $\varphi$ 
by means of equation \eqref{E:5}. It follows that $\varphi$ is infinitely differentiable 
and rapidly decreasing. Since $\widehat\varphi$ satisfies \eqref{E:2} with $k=2$, we obtain 
that $\varphi$ satisfies condition (d) with $k=2$.  We will show that $\varphi$ also satisfies
conditions (a), (b) and (c). Instead of using Paley-Wiener's Theorem we prefer 
to use another method,
which gives us some additional information.

Let $\mu_n$ be the Radon measure in $\R$ whose Fourier transform is
\begin{equation}\label{E:10}
\mathcal F(\mu_m)=\prod_{k=1}^m\Bigl(\cos\frac{\pi x}{2^k}\Bigr)^k.
\end{equation}

Since
\begin{equation}\label{E:11}
\mathcal F\bigl(\tfrac12\delta_{2^{-k-1}}+\tfrac12\delta_{-2^{-k-1}}
\bigr)=\cos\frac{\pi x}{2^k},
\end{equation}
$\mu_m$ is the convolution product 
\begin{equation}\label{E:12}
\mu_m=\Conv_{k=1}^\infty\bigl(\tfrac12\delta_{2^{-k-1}}+\tfrac12\delta_{-2^{-k-1}}
\bigr)^k
\end{equation}
where the powers have also the meaning of convolution products. 

It is clear that the total variation $\Vert \mu_m\Vert=1$, $\mu_m\ge0$ and 
$\supp(\mu_m)\subset[-1,1]$. The last assertion follows from
\[\sum_{k=1}^\infty \frac{k}{2^{k+1}}=1.\]

\begin{lem}\label{L:1}
Let  $(\mu_m)$ be the sequence of measures defined in \eqref{E:12}. This
sequence of measures converges in the weak-* topology $\sigma(\mathcal M_b(\R),C^*(\R))$
towards the measure $\varphi\lambda$ with density $\varphi$ with 
respect to Lebesgue measure $\lambda$.
\end{lem}

\begin{proof} Denote by
$C^*(\R)$ the Banach space of complex valued bounded functions defined on $\R$.
Since the measures $\mu_m$ are on the unit ball of the dual space, which is weakly compact, 
there is a measure $\mu$ that is a weak cluster point to the sequence $\mu_m$. Since
$\mathcal F(\mu_m)\to \mathcal F(\varphi\lambda)$ pointwise, we have $\mathcal F(\mu)=
\mathcal F(\varphi\lambda)$. Since $\mathcal F$ is injective in the space of bounded 
Radon measures, we obtain $\mu=\varphi\lambda$. Therefore $\varphi\lambda$ is the only 
weak cluster point, so that it is the weak limit of the sequence $\mu_m$.
\end{proof}

Since $\mu_m\to\varphi\lambda$ with weak convergence, it follows that $\varphi$ satisfies 
condition (a) and, since $\varphi$ is continuous it follows that $\varphi(x)\ge0$ for 
all $x\in\R$. 

Now we know that $\int\varphi(t)\,dt=\widehat\varphi(0)=1$. This fact, together 
with the fact that 
$\supp(\varphi)=[-1,1]$ yields
\begin{multline*}
\varphi(0)=\int_{-1}^0\varphi'(t)\,dt=\int_{-1}^02\bigl(\varphi(2t+1)-\varphi(2t-1)\bigr)\,dt\\
=2\int\varphi(2t+1)\,dt=\int\varphi(u)\,du=1.
\end{multline*}
and $\varphi$ satisfies condition (c).

It remains to show that $\varphi$ satisfies (b). By the same reasoning as above we have 
for every $x\in(-1,0)$
\begin{equation}\label{E:13}
\varphi(x)=2\int_{-1}^x \varphi(2t+1)\,dt.
\end{equation}
Therefore $\varphi(x)$ is not decreasing in $(-1,0)$ ( since $\varphi'(x)\ge0$ ).
Since $\varphi$ is an even function, $\varphi(x)>0$ implies $\varphi(t)>0$ for all 
$t\in(-x,x)$. If $\varphi(x)>0$ we have $\varphi((x-1)/2)>0$, therefore $\varphi(t)>0$
for $t\in(-1,1)$. 
\end{proof}

\section{Other expressions for $\varphi$.}
We have seen two possible definitions of $\varphi$: the expression \eqref{E:5} and 
that given in Lemma \ref{L:1}. We will give another two. One as the limit of a sequence of
step functions and another by means of an integral. We need some previous notations and
definitions.

Let $p_n$ be the sequence of polynomials defined by the recurrence
\begin{equation}\label{E:14}
p_0=1;\qquad p_n(x)=p_{n-1}(x^2)(1+x)^n.
\end{equation}
It is easy to see that 
\begin{equation}\label{E:15}
p_n(x)=\prod_{k=1}^n\Bigl(\frac{1-x^{2^k}}{1-x}\Bigr)
\end{equation}
The degree $g_n$ of $p_n$ is given by the equations
\begin{equation}\label{E:16}
g_0=0,\quad g_n=2g_{n-1}+n.
\end{equation}
Therefore
\begin{equation}\label{E:17}
\frac{g_n}{2^n}=\frac12+\frac{2}{2^2}+\cdots+\frac{n}{2^n}.
\end{equation}
Equations \eqref{E:12} and \eqref{E:14} show that $\mu_n$ is the measure
obtained when we substitute each power $x^m$ by $\delta_{\frac{2m-g_n}{2^{n+1}}}$ in the polynomial
\[2^{-\binom{n+1}{2}}p_n(x).\]
For each $n\in\N$, let $\varphi_n$ be the step function obtained from
the polynomial $2^{-\binom{n+1}{2}}p_n(x)$ substituting each power $x^m$ by the characteristic
function of the interval 
\[ \Bigl[\frac{2m-1-g_n}{2^{n+1}},\frac{2m+1-g_n}{2^{n+1}}\Bigr]\]
multiplied by $2^n$. We have then:
\begin{thm}\label{T:2}
$\varphi$ is the limit of the sequence of step functions $\varphi_m$.
\end{thm}
\begin{proof}
It suffices to observe that for a characteristic function $f$ of an interval with dyadic 
extremes, we have
\[\lim_{m\to\infty}\mu_m(f)=\lim_{m\to\infty}\int\varphi_m f=\int\varphi f,\]
and the fact, easily proved, that $\varphi_m$ is monotonous non decreasing in $(-1,0)$
and monotonous not increasing in $(0,1)$, and that $\varphi_m(0)=1$. 
\end{proof}

It is easy to see that 
\begin{equation}\label{E:18}
p_{m+1}(x)=p_m(x)(1+x+x^2+\cdots+x^{2^{m+1}-1})
\end{equation}
This gives us an easy algorithm to obtain the $\varphi_m$, and also shows that
\begin{equation}\label{E:19}
p_{m}(x)=(1+x)(1+x+x^2+x^3)\cdots(1+x+\cdots+x^{2^m-1}).
\end{equation}
Therefore we have a combinatorial interpretation  of the coefficient of $x^r$ in $p_m(x)$:

\emph{The coefficient of $x^r$ in $p_m(x)$ is the number of partitions of $r$, in $m$ parts
$r=s_1+s_2+\cdots+s_m$ such that $0\le s_i\le 2^i-1$}.

\begin{thm}\label{T:3}
Let $\sigma=\bigotimes_{k=1}^\infty \lambda_k$ be the measure defined on $[0,1]^\N$, $\lambda_k$
being the Lebesgue measure on $[0,1]$. For $-1\le x\le 0$ we have
\[\varphi(x)=\sigma\Bigl\{(x_k)\colon 0\le \sum_{k=1}^\infty\frac{x_k}{2^k}\le x+1\Bigr\}\]
\end{thm}
\begin{proof}
Let $\nu_k$ be the measure in $[-1,1]^\N$
\[\nu_k=\bigotimes_{m=1}^\infty \bigl(\frac12\delta_{2^{-m-k}}+\frac12\delta_{-2^{-m-k}}\bigr)\]
($k=1$ $2$, \dots, ) and let $(t_{k,1}, t_{k,2},\dots)$ denote the variables in the space
$[-1,1]^\N$.

Let $\mu$ be the measure defined on $\{0,1\}^\N$ as the product of the measure 
assigning $0$ and $1$ measure $1/2$. 

Then $\nu_k=f_k(\mu)$ the image measure, with $f_k\{0,1\}^\N\to[-1,1]^\N$ given by 
$f_k(\varepsilon_1,\varepsilon_2,\dots)=(t_{k,1}, t_{k,2},\dots)$ where
\[t_{k,m}=\begin{cases}
2^{-m-k} & \text{when $\varepsilon_m=1$},\\
-2^{-m-k}  & \text{when $\varepsilon_m=0$}.
\end{cases}\]
$\mu$ is also the image measure of Lebesgue measure on $[0,1]$ by the application 
$g\colon[0,1]\to \{0,1\}^\N$ defined by $g(x)=(\varepsilon_1,\varepsilon_2,\dots)$
if $x=\sum_{m=1}^\infty (\varepsilon_m/2^m)$ with $\varepsilon_m\in\{0,1\}$. The function
$g$ is well defined only almost everywhere but this is no difficulty. 

The measure $\varphi(t)\,dt$ is the limit of the $\mu_m$, therefore for all integrable
$f$, 
\[\int f(t)\varphi(t)\,dt=\int f\Bigl(\sum t_{k,m}\Bigr)\,d\bigotimes_{k=1}^\infty \nu_k.\]
Since each $\nu_k$ is an image measure the last integral can be transformed in an 
integral on $[0,1]^\N$ with respect to the measure $\sigma=\otimes _{k=1}^\infty\lambda$.

The relation $f_k\circ g(x_k)=(t_{k,1},t_{k,2},\dots)$ implies 
$x_k=\sum_{m=1}^\infty (\varepsilon_m/2^m)$ with $\varepsilon_m\in\{0,1\}$, 
$t_{k,m}=2^{-m-k}$ if $\varepsilon_m=1$ and $t_{k,m}=-2^{-m-k}$ when $\varepsilon_m=0$.
Therefore
\[\sum_m t_{k,m}=\sum_{m=1}^\infty \varepsilon_m2^{-m-k}-
\Bigl(\sum_{m=1}^\infty 2^{-m-k}-\sum_{m=1}^\infty \varepsilon_m2^{-m-k}\Bigr)=
x_k2^{-k+1}-2^{-k}\]
From this we get
\[\int f(t)\varphi(t)\,dt=\int f\Bigl(\sum_{k=1}^\infty x_k2^{-k+1}-1\Bigr)\,d\sigma.\]
Taking $f(t)=\chi_{[-1,2x+1]}(t)$ with $-1\le x\le 0$,
\begin{multline}\label{E:20}
\varphi(x) =\int_{-1\le \sum_{k=1}^\infty x_k2^{-k+1}-1\le 2x+1} d\sigma=
\int_{0\le \sum_{k=1}^\infty x_k2^{-k}\le x+1} d\sigma\\=
\sigma\Bigl\{(x_k)\colon 0\le \sum_{k=1}^\infty x_k2^{-k}\le x+1\Bigr\}
\end{multline}

In other words we have proved the Proposition: \emph{Let $x_k$ be independent random variables  uniformly distributed in 
$[0,1]$, $\varphi(x)$ (with $-1\le x\le 0$) is equal to 
the probability that the sum $\sum x_k2^{-k}$  be $\le x+1$}. 
\end{proof}

\section{Properties.}
\begin{thm}\label{T:4}
Let 
\[\theta(t)=\sum_{k=0}^\infty (-1)^{s(k)}\varphi(t-2k-1)\]
where $s(k)$ denotes the sum of the digits of $k$ when written in base 2.
Then
\begin{itemize}
\item[(a)] $\theta$ is an infinitely differentiable function.
\item[(b)] $\theta'(t)=2\theta(2t)$.
\item[(c)] For $t\in[-1,1]$, $\varphi^{(k)}(t)=2^{\binom{k+1}{2}}\theta(2^k t+2^k)$.
\end{itemize}
\end{thm}
\begin{proof}
The sum in the definition of $\theta(t)$ is locally finite, therefore $\theta$ is 
infinitely differentiable and its derivative is
\begin{multline*}
\theta'(t)=\sum_{k=0}^\infty(-1)^{s(k)}2
\bigl(\varphi(2t-4k-2+1)-\varphi(2t-4k-2-1)\bigr)\\=
2\sum_{k=0}^\infty\bigl((-1)^{s(k)}\varphi(2t-2(2k)-1)
-(-1)^{s(k)}\varphi(2t-2(2k+1)-1)\bigr)
\end{multline*}
\begin{figure}[H]
\centering
\includegraphics[width=\hsize]{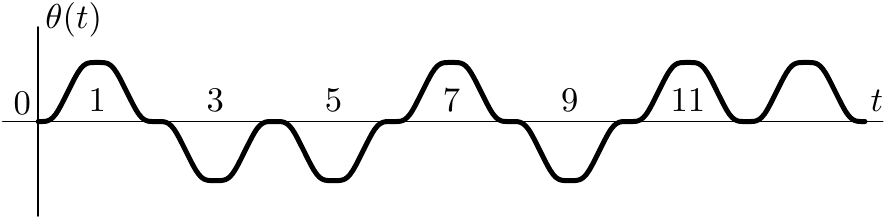}
\end{figure}
\noindent using the definition of $s(k)$ this yields
\begin{equation}\label{E:21}
\theta'(t)=2\theta(2t).
\end{equation}
By repeated differentiation of \eqref{E:21} we obtain
\begin{equation}\label{E:22}
\theta^{(k)}(t)=2^{\binom{k+1}{2}}\theta(2^kt).
\end{equation}
For $t\in[-1,1]$ we have $\varphi(t)=\theta(t+1)$ so that
\begin{equation}\label{E:23}
\varphi^{(k)}(t)=2^{\binom{k+1}{2}}\theta(2^kt+2^k),\qquad \text{if}\quad t\in[-1,1].
\end{equation}
\end{proof}
This proves that on any dyadic point $t=q/2^n$ the Taylor expansion is a polynomial
\begin{equation}\label{E:24}
T(t,x)=\sum_{k=0}^n \frac{\varphi^{(k)}(t)}{k!} x^k
\end{equation}
and for $q$  odd the degree of $T(t,x)$ is $n$. 
\begin{cor*}
The function $\varphi$ is not analytic on any point of the interval $[-1,1]$.
\end{cor*}
\begin{thm}\label{T:5}
For $u>0$ and $t\in\R$ we have
\begin{equation}\label{E:25}
\sum_{k\in\Z}\varphi(t+uk)=\sum_{k\in\Z}\frac{1}{u}
\widehat\varphi\Bigl(\frac{k}{u}\Bigr)e^{2\pi i k\frac1u}.
\end{equation}
\end{thm}
\begin{proof}
The left hand side of \eqref{E:25} is locally finite, therefore the sum is 
infinitely differentiable. It is a periodic function of $t$ with period $u$. Therefore
it has a Fourier series expansion
\[\sum_{k\in\Z}\varphi(t+uk)=\sum_{k\in\Z}a_k e^{2\pi i k\frac1u}\]
where
\begin{multline*}
a_n=\frac{1}{u}\int_0^u\sum_{k\in\Z}\varphi(t+uk)e^{-2\pi i n\frac{t}{u}}\,dt=
\sum_{k\in\Z}\frac{1}{u}\int_0^u\varphi(t+uk)e^{-2\pi i n\frac{t}{u}}\,dt\\=
\sum_{k\in\Z}\frac{1}{u}\int_{uk}^{u(k+1)}\varphi(v)e^{-2\pi i n\frac{v-uk}{u}}\,dv=
\frac{1}{u}\int\varphi(v)e^{-2\pi i v\frac{n}{u}}\,dv=
\frac{1}{u}\widehat\varphi\Bigl(\frac{n}{u}\Bigr).
\end{multline*}
\end{proof}

Some particular cases of \eqref{E:25} are interesting:
\begin{equation}\label{E:26}
\sum_{k\in\Z}\varphi\Bigl(t+\frac{k}{n}\Bigr)=n \qquad \text{for }\quad n\in\N.
\end{equation}
Furthermore
\begin{equation}\label{E:27}
\sum_{k\in\Z}\varphi(t+k)=1.
\end{equation}
which is equivalent to 
\begin{equation}\label{E:28}
\varphi(t)+\varphi(t-1)=1,\qquad \text{for }\quad t\in[0,1].
\end{equation}

Also, from \eqref{E:25} it follows that
\begin{equation}\label{E:29}
\sum_{k\in\Z}\varphi(t+2k)=\frac12\sum_{k\in\Z}
\widehat\varphi\Bigl(\frac{k}{2}\Bigr)e^{\pi i k t},
\end{equation}
which  is no more than the Fourier expansion
\begin{equation}\label{E:30}
\varphi(t)=\frac12+\sum_{k=0}^\infty \widehat\varphi\Bigl(\frac{2k+1}{2}\Bigr)
\cos(2k+1)\pi t,
\end{equation}
valid for $t\in[-1,1]$ and which has good convergence properties.

The product \eqref{E:9} implies that the sign of the coefficient $\widehat\varphi((2k+1)/2)$
is the parity of $1+v_2(1)+1+v_2(2)+\cdots+1+v_2(k)=k+v_2(k!)=s(k)$, therefore also equal
to the sign of $\theta(k)$.

Equation \eqref{E:25} is not only a Fourier expansion, it is also Poisson's formula applied 
to $\varphi(t+x)$. For $t=0$ it yields
\begin{equation}\label{E:31}
\sum_{m\in\Z}\varphi(ma)=\sum_{m\in\Z}\frac{1}{a}\widehat\varphi\Bigl(\frac{m}{a}\Bigr),
\end{equation}
and using the knowledge about the support of $\varphi$, this implies
\begin{equation}\label{E:32}
a+2a\varphi(a)=\sum_{m\in\Z}\frac{1}{a}\widehat\varphi\Bigl(\frac{m}{a}\Bigr),
\qquad \text{for }\quad\frac12\le a\le 1.
\end{equation}

\section{Values at dyadic points.}
First we determine the values of $\varphi(1-2^{-n})$.
\begin{thm}\label{T:6}
For each natural number $n$ we have
\begin{equation}\label{E:33}
\int_0^1t^{n-1}\varphi(t)\,dt=(n-1)!\,2^{\binom{n}{2}}\,\varphi(1-2^{-n}).
\end{equation}
\begin{equation}\label{E:34}
\int_0^1t^{2n}\varphi(t)\,dt=\frac{c_n}{2}.
\end{equation}
where $c_n$ are the rational numbers that appear in the expansion \eqref{E:6} of 
$\varphi$.
\end{thm}
\begin{proof}
We can check, by differentiation, that in the sequence of functions 
\begin{gather*}
f_0(t)=\varphi(t),\quad f_1(t)=\varphi\Bigl(\frac{t}{2}-\frac{1}{2}\Bigr),\quad
f_2(t)=2\,\varphi\Bigl(\frac{t}{4}-\frac{1}{4}-\frac{1}{2}\Bigr),\\
f_k(t)=2^{\binom{k}{2}}\,\varphi\Bigl(\frac{t}{2^k}-\frac{1}{2^k}-\frac{1}{2^{k-1}}-
\cdots-\frac{1}{2}\Bigr)
\end{gather*}
each function is a primitive in $[-1,1]$ of the previous one and all vanish at the point
$t=-1$. So integrating by parts
\begin{multline*}
\int_0^1t^n\varphi(t)\,dt=(-1)^n \int_{-1}^0t^n\varphi(t)\,dt=
(-1)^n\int_{-1}^0 t^nf_0(t)\,dt\\
=-(-1)^n n \int_{-1}^0 t^{n-1}f_1(t)\,dt=
(-1)^n(-1)^n n! \int_{-1}^0 f_n(t)\,dt\\=n! f_{n+1}(0)=n!\, 2^{\binom{n+1}{2}}\,
\varphi(1-2^{-n-1}).
\end{multline*}
Moreover
\[\sum_{n=0}^\infty \frac{x^n}{n!}\int_{-1}^{+1}t^n\varphi(t)\,dt=
\int_{-1}^{+1}e^{xt}\varphi(t)\,dt=\widehat\varphi\Bigl(\frac{ix}{2\pi}\Bigr)=
\sum_{k=0}^\infty \frac{c_k}{(2k)!}x^{2k},\]
and this proves \eqref{E:34}.
\end{proof}
From the two formulas we obtain
\begin{equation}\label{E:35}
\varphi(1-2^{-2n-1})=\frac{2^{-\binom{2n+1}{2}}}{2 (2n)!}\frac{F_n}{(2n+1)(2n-1)\cdots1}
\prod_{k=1}^n(2^{2k}-1)^{-1},
\end{equation}
where $F_k$ are the integers defined in \eqref{E:8}.

We may compute in a similar way all the numbers $\varphi(1-2^{-n})$. With this objective
notice that
\begin{multline*}
\int_0^1\varphi(t)e^{-2\pi i xt}\,dt=\frac{1}{2\pi i x}+\int_0^1\varphi'(t)
\frac{e^{-2\pi i x t}}{2\pi i x}\,dt\\
=\frac{1}{2\pi i x}-\int_0^1 2\varphi(2t-1)
\frac{e^{-2\pi i x t}}{2\pi i x}\,dt
=\frac{1}{2\pi i x}\Bigl(1-e^{-\pi i x}\widehat\varphi\Bigl(\frac{x}{2}\Bigr)\Bigr).
\end{multline*}
Therefore
\begin{equation}\label{E:36}
\int_0^1e^{xt}\varphi(t)\,dt=\sum_{n=0}^\infty \frac{x^n}{n!}\int_0^1t^n\varphi(t)\,dt
=-\frac{1}{x}\Bigl(1-e^{\frac{x}{2}}
\widehat\varphi\Bigl(\frac{ix}{4\pi}\Bigr)\Bigr)
\end{equation}
from which we obtain $\varphi(1-2^{-n})$. Another way to compute these numbers is
to use 
\begin{equation}\label{E:37}
f(x)=1+x\int_0^1 e^{xt}\varphi(t)\,dt=e^{\frac{x}{2}}
\widehat\varphi\Bigl(\frac{ix}{4\pi}\Bigr),
\end{equation}
together with the fact that
\begin{equation}\label{E:38}
f(2x)=\frac{e^x-1}{x}f(x).
\end{equation}
Therefore 
\begin{equation}\label{E:39}
f(x)=\sum_{n=0}^\infty \frac{d_n}{n!}x^n,
\end{equation}
where $d_0=1$ and we have the recurrence
\begin{equation}\label{E:40}
(n+1)(2^n-1)d_n=\sum_{k=0}^{n-1}\binom{n+1}{k}d_k.
\end{equation}
It follows that there are integers $G_n$ such that
\begin{equation}\label{E:41}
d_n=\frac{G_n}{(n+1)!}\prod_{k=1}^n(2^k-1)^{-1}.
\end{equation}
The numbers $d_n$, equation \eqref{E:33} and
\begin{equation}\label{E:42}
d_n=n\int_0^1t^{n-1}\varphi(t)\,dt
\end{equation}
determine the values of $\varphi(1-2^{-n})$.

We may prove now the following theorem:
\begin{thm}\label{T:7}
The function $\varphi$ takes rational values at each dyadic point.
\end{thm}
\begin{proof}
Let $t=q/2^n$ with $|q|<2^n$. We compute $\varphi(q2^{-n})$.
Since $\varphi$ and all its derivatives vanish at the point $-1$,
Taylor's theorem with the rest in integral form gives us
\[\varphi(q2^{-n})=\int_{-1}^t \frac{(t-x)^n}{n!}\varphi^{(n+1)}(x)\,dx.\]
Applying our formula for the derivatives of $\varphi$  we obtain
\[\varphi(t)=\frac{1}{n!}2^{\binom{n+2}{2}}\int_{-1}^t(t-x)^n\theta(2^{n+1}(1+x))\,dx.\]
Since for $2h\le 2^{n+1}(1+x)\le 2(h+1)$ we have \[\theta(2^{2n+1}(1+x))=(-1)^{s(h)}
\varphi(2^{n+1}(1+x)-2h-1)\] and putting $2^{n+1}(1+x)-2h-1=u$ we obtain
\begin{multline*}
\varphi(t)=\frac{1}{n!}2^{\binom{n+2}{2}}2^{-n-1}\sum_{h=0}^{q+2^n-1}(-1)^{s(h)}
\int_{-1}^1\Bigl(t-\frac{u}{2^{n+1}}-\frac{2h+1}{2^{n+1}}+1\Bigr)^n\varphi(u)\,du\\
=\frac{1}{n!}2^{-\binom{n+1}{2}}\sum_{h=0}^{q+2^n-1}(-1)^{s(h)}
\int_{-1}^1\bigl(2(q-h)+2^{n+1}-1-u\bigr)^n\varphi(u)\,du\\
=\frac{1}{n!}2^{-\binom{n+1}{2}}\sum_{h=0}^{q+2^n-1}(-1)^{s(h)}\sum_{k=0}^n
\binom{n}{k}\bigl(2(q-h)+2^{n+1}-1\bigr)^{n-k}(-1)^k\int_{\R}u^k\varphi(u)\,du.
\end{multline*}
This formula, together with  equality
\[\int_{-1}^1u^n\varphi(u)\,du=\bigl(1+(-1)^n\bigr)\int_0^1u^n\varphi(u)\,du\]
and \eqref{E:34} proves our theorem, and we obtain
\[\varphi(q2^{-n})=
2\sum_{h=0}^{q+2^n-1}\sum_{k=0}^{\lfloor n/2\rfloor}(-1)^{s(h)}
\frac{2^{\binom{2k+1}{2}-\binom{n+1}{2}}}{(n-2k)!}
\bigl(2(q-h)+2^{n+1}-1\bigr)^{n-2k}\varphi(1-2^{-2k-1})\]
\end{proof}

For the computation we may  first obtain the common denominator of $\varphi(q2^{-n})$
for a fixed $n$, and using \eqref{E:30} it is possible then to compute the exact
value of $\varphi(q2^{-n})$. For $n=5$ the common denominator is $33\,177\,600=
2^{14}3^45^2$ and we obtain

\[\begin{array}{|r|r||r|r||r|r|}
\hline
q & 33\,177\,600\,\varphi(q/32) & q & 33\,177\,600\,\varphi(q/32) & q & 33\,177\,600\,\varphi(q/32) \\ \hline
0 & 33\,177\,600 & 11 & 26\,622\,019 & 22 & 4\,893\,712\\
1 & 33\,177\,581 & 12 & 24\,768\,000 & 23 & 3\,470\,381\\
2 & 33\,175\,312 & 13 & 22\,784\,381 & 24 & 2\,304\,000\\
3 & 33\,152\,381 & 14 & 20\,733\,712 & 25 & 1\,396\,781\\
4 & 33\,062\,400 & 15 & 18\,662\,381 & 26 & 746\,512\\
5 & 32\,842\,819 & 16 & 16\,588\,800 & 27 & 334\,781\\
6 & 32\,431\,088 & 17 & 14\,515\,219 & 28 & 115\,200\\
7 & 31\,780\,819 & 18 & 12\,443\,888 & 29 & 25\,219\\
8 & 30\,873\,600 & 19 & 10\,393\,219 & 30 & 2\,288\\
9 & 29\,707\,219 & 20 & 8\,409\,600 &  31 & 19\\
10 & 28\,283\,888 & 21 & 6\,555\,581 & 32 & 0\\ \hline
\end{array}\]


\begin{bibdiv}
\begin{biblist}

\bib{MR0151354}{book}{
   author={Bourbaki, N.},
   title={Fonctions d'une variable r\'eelle },
   publisher={Hermann, Paris},
   date={1958},
}

\bib{MR0188387}{book}{
   author={Hewitt, E.},
   author={Stromberg, K.},
   title={Real and abstract analysis},
   publisher={Springer-Verlag, Berlin},
   date={1965},
}

\bib{MR0205028}{book}{
   author={Horv{\'a}th, J.},
   title={Topological vector spaces and distributions. Vol. I},
   publisher={Addison-Wesley Publishing Co., Reading, Massachusetts},
   date={1966},
}

\bib{MR1501802}{article}{
   author={Jessen, B.},
   author={Wintner, A.},
   title={Distribution functions and the Riemann zeta function},
   journal={Trans. Amer. Math. Soc.},
   volume={38},
   date={1935},
   pages={48--88},
}



\end{biblist}
\end{bibdiv}
\end{document}